\documentclass[a4paper]{article}

\usepackage[T1]{fontenc}
\usepackage[utf8]{inputenc}
\usepackage[pass]{geometry}
\usepackage{tikz}
	\usetikzlibrary{positioning}
\usepackage{amsmath}
\usepackage{amsfonts}
\usepackage{amssymb}
\usepackage{mathtools}
\usepackage{amsthm}
\usepackage{etoolbox}
\usepackage{caption}
\RequirePackage{hyperref}
	\hypersetup{linktoc=all,hidelinks}
\usepackage[capitalize,nameinlink]{cleveref}
\usepackage[noadjust]{cite}

\theoremstyle{definition}
\newtheorem{theorem}{Theorem}[section]
\newtheorem{proposition}[theorem]{Proposition}
\newtheorem{corollary}[theorem]{Corollary}
\newtheorem{property}{Property}[section]
\newtheorem{conjecture}{Conjecture}
\newtheorem{claim}{Claim}
\newtheorem{case}{Case}[theorem]
\newtheorem{example}{Example}[section]
\newtheorem{problem}{Problem}
\crefname{claim}{Claim}{Claims}
\crefname{case}{Case}{Cases}
\crefname{conjecture}{Conjecture}{Conjectures}
\AfterEndEnvironment{case}{\noindent\ignorespaces}

\let\eqref\labelcref
\crefname{equation}{}{}

\crefname{enumi}{}{}

\makeatletter
\newenvironment{boldproof}[1][\proofname]{\par
  \pushQED{\qed}%
  \normalfont \topsep6\p@\@plus6\p@\relax
  \trivlist
  \item[\hskip\labelsep
        \bfseries
    #1\@addpunct{.}]\ignorespaces
}{%
  \popQED\endtrivlist\@endpefalse
}
\makeatother

\newcommand{\dotsp}{\cdots}
\newcommand{\dir}[1]{\smash{\vec{#1}}}
\newlength{\revdirraise}
\newlength{\revdirextraraise}
\newcommand{\revdir}[1]{\smash{%
	\settoheight{\revdirraise}{$#1$}%
	\settodepth{\revdirextraraise}{$#1$}%
	\ooalign{$#1$\cr{\raisebox{\revdirraise+\revdirextraraise+0.85pt}%
		{\rlap{$\mkern4.1mu$\rotatebox[origin=c]{180}%
			{$\vec{\phantom #1}$}}}}}}}

\DeclarePairedDelimiter{\card}{\lvert}{\rvert}
\DeclarePairedDelimiter{\set}{\lbrace}{\rbrace}
\DeclarePairedDelimiter{\induced}{<}{>}
\DeclarePairedDelimiter{\indset}{{<}\{}{\}{>}}

\AtBeginDocument{\frenchspacing}

\newcommand{\internalvertex}[2]{node [circle, inner sep=0pt, outer sep=0pt, minimum size=4pt, fill=#1, draw=#2] {}}
\newcommand{\vertex}{\internalvertex{black}{black}}
\newcommand{\point}{\node [inner sep=0pt,outer sep=0pt,minimum size=0pt]}
\newcommand{\insidepoint}{node [inner sep=0pt,outer sep=0pt,minimum size=0pt]}
\newcommand{\ellipsisdot}{node [circle, inner sep=0pt, outer sep=0pt, minimum size=2.5pt, fill=black, draw=black] {}}

\newcommand{\figunihamnohamcircleoneend}{%
\begin{tikzpicture}
\foreach \x in {0,2,4}{
	\draw (\x,1) -- (\x,4) to [out=240,in=120] (\x,1); 
	\foreach \y in {1,2}{
		\draw (\x,\y) to [out=110,in=250] (\x,{\y+2});}}
\foreach \x/\y in {1/-1,2/-1,3/1,4/1}{
	\draw (0,\x) \vertex
	   -- (1,{\x+0.25*\y}) \vertex
	   -- (2,\x) \vertex
	   -- (3,{\x+0.25*\y}) \vertex
	   -- (4,\x) \vertex;
	\foreach \y in {1,2,3,4}{
		\draw (0,\x) -- (2,\y);
		\draw (2,\x) -- (4,\y);}}
\foreach \z in {1.5,3.5}{
	\draw (0,{\z-0.5})
	   -- (-1,\z) \vertex
	   -- (0,{\z+0.5});}
\foreach \z in {0.2,0.45,0.7}
	\fill ({4+\z},2.5) \ellipsisdot;
\end{tikzpicture}}

\newcommand{\figotherhamballsnonham}{%
\begin{tikzpicture}
\foreach \x in {0,2,4}{
	\draw (\x,1) -- (\x,4) to [out=240,in=120] (\x,1); 
	\foreach \y in {1,2}{
		\draw (\x,\y) to [out=110,in=250] (\x,{\y+2});}}
\foreach \x/\y in {1/-1,2/-1,3/1,4/1}{
	\draw (0,\x) \vertex
	   -- (1,{\x+0.25*\y}) \vertex
	   -- (2,\x) \vertex
	   -- (3,{\x+0.25*\y}) \vertex
	   -- (4,\x) \vertex;
	\foreach \y in {1,2,3,4}{
		\draw (0,\x) -- (2,\y);
		\draw (2,\x) -- (4,\y);}}
\foreach \x/\y in {0/-1,4/1}{
	\foreach \z in {1.5,3.5}{
		\draw (\x,{\z-0.5})
		   -- (\x+\y,\z) \vertex
		   -- (\x,{\z+0.5});}}
\draw
	(-1,3.5 ) node [anchor=south east] {$a_{ 1}$}
	(-1,1.5 ) node [anchor=south east] {$a_{ 2}$}
	( 5,3.5 ) node [anchor=south west] {$a_{ 3}$}
	( 5,1.5 ) node [anchor=south west] {$a_{ 4}$}
	( 0,4   ) node [anchor=south east] {$b_{ 1}$}
	( 0,1   ) node [anchor=north east] {$b_{ 4}$}
	( 4,4   ) node [anchor=south west] {$b_{ 9}$}
	( 4,1   ) node [anchor=north west] {$b_{12}$}
	( 2,4   ) node [anchor=south     ] {$b_{ 5}$}
	( 2,1   ) node [anchor=north     ] {$b_{ 8}$}
	( 1,4.25) node [anchor=south     ] {$c_{ 1}$}
	( 1,0.75) node [anchor=north     ] {$c_{ 4}$\vphantom{b}}
	( 3,4.25) node [anchor=south     ] {$c_{ 5}$}
	( 3,0.75) node [anchor=north     ] {$c_{ 8}$\vphantom{b}}
	;
\end{tikzpicture}}

\newcommand{\figdoubleraysquareinner}{%
\draw (0,0) \vertex;
\newcommand{\figdoubleraysquarelength}{3}
\foreach \x in {0,...,\figdoubleraysquarelength}{
	\draw  (\x+1,1)
		-- (\x,1) \vertex
		-- (\x,0) \vertex
		-- (\x+1,0)
		-- (\x,1);}
\draw  ({\figdoubleraysquarelength+1},0) \vertex
	-- ({\figdoubleraysquarelength+1},1) \vertex;}
\newcommand{\figdoubleraysquare}{%
\begin{tikzpicture}[baseline=0pt]
\figdoubleraysquareinner
\foreach \z in {0.2,0.45,0.7}{
	\fill ({-\z},0.5) \ellipsisdot;
	\fill ({\figdoubleraysquarelength+1+\z},0.5) \ellipsisdot;}
\end{tikzpicture}}

\newcommand{\figclawarmbase}[1]{%
\foreach \x in {1,...,#1}
	\point (v\x) at (1,{-1+3*(\x-1)/(#1-1)}) {};
\foreach \x in {-.5,...,1.5}{
	\foreach \y in {1,...,#1}{
		\draw (0,\x) -- (v\y);}}
\foreach \x in {1,...,#1}{
	\foreach \y in {0,1}{
		\draw (v\x) -- (2,\y);}}
\foreach \x in {-.5,...,1.5}
	\draw (0,\x) \vertex;
\foreach \x in {1,...,#1}
	\draw (v\x) \vertex;
\draw [dashed,rounded corners=0.3cm]
	(0.7,-1.3) rectangle  (1.3,2.3);
\node at (1,-1.3) [below=2pt] {$K_{#1}$};}

\newcommand{\figclawarmore}{%
\begin{tikzpicture}[baseline=0pt]
\figclawarmbase{5}
\foreach \x in {2,3}{
	\foreach \y in {0,1}{
		\foreach \z in {0,1}{
			\draw (\x,\y) -- ({\x+1},\z);}}}
\foreach \x in {2,3,4}
	\draw (\x,0) -- (\x,1);
\foreach \x in {2,3,4}{
	\foreach \y in {0,1}{
		\draw (\x,\y) \vertex;}}
\foreach \z in {0.2,0.45,0.7}
	\fill ({4+\z},0.5) \ellipsisdot;
\end{tikzpicture}}

\newcommand{\figclawarmfinitefive}{%
\begin{tikzpicture}[baseline=0pt]
\figclawarmbase{5}
\foreach \x in {2,3}{
	\foreach \y in {0,1}{
		\foreach \z in {0,1}{
			\draw (\x,\y) -- ({\x+1},\z);}}}
\foreach \x in {2,3,4}
	\draw (\x,0) -- (\x,1);
\foreach \x in {2,3,4}{
	\foreach \y in {0,1}{
		\draw (\x,\y) \vertex;}}
\end{tikzpicture}}

\newcommand{\figMtwoonlyfinite}{%
\begin{tikzpicture}[baseline=0pt]
\figclawarmbase{6}
\foreach \x in {2,...,4}{
	\foreach \y in {0,1}{
		\foreach \z in {0,1}{
			\draw (\x,\y) -- ({\x+1},\z);}}}
\foreach \x in {2,...,5}
	\draw (\x,0) -- (\x,1);
\draw
	(5.5,1.5) -- (4,1)
	(5.5,1.5) -- (5,1)
	(5.5,1.5) -- (5,0)
	(5.5,1.5) \vertex;
\foreach \x in {2,...,5}{
	\foreach \y in {0,1}{
		\draw (\x,\y) \vertex;}}
\end{tikzpicture}}

\newcommand{\figMtwoonlyinfinite}{%
\begin{tikzpicture}[baseline=0pt]
\figclawarmbase{6}
\foreach \x in {2,...,8}{
	\foreach \y in {0,1}{
		\foreach \z in {0,1}{
			\draw (\x,\y) -- ({\x+1},\z);}}}
\foreach \x in {2,...,9}
	\draw (\x,0) -- (\x,1);
\draw
	(5.5,1.5) -- (4,1)
	(5.5,1.5) -- (5,1)
	(5.5,1.5) -- (5,0)
	(5.5,1.5) -- (6,1)
	(5.5,1.5) -- (6,0)
	(5.5,1.5) \vertex;
\foreach \x in {2,...,9}{
	\foreach \y in {0,1}{
		\draw (\x,\y) \vertex;}}
\foreach \z in {0.2,0.45,0.7}
	\fill[overlay] ({9+\z},0.5) \ellipsisdot;
\end{tikzpicture}}

\newcommand{\figultrametaclaw}{%
\begin{tikzpicture}
\foreach \x in {0,...,8}
	\point (\x) at ({40*\x+140}:1.1) {};
\foreach \x/\y in {0/1,1/2,2/3,3/4,4/5,5/6,6/7,7/8}{%
	\foreach \z in {\y,...,8}{%
		\draw (\x) -- (\z);}}
\foreach \x/\y in {0/0,1/1,2/2}
	\point (a0\x) at (\y) {};
\foreach \x in {1,2,3}{%
	\foreach \y in {0,1}{%
		\point (a\x\y) at ({-\x-.8},{\y-.5}) {};}}
\foreach \x/\y in {0/3,1/4,2/5}
	\point (b0\x) at (\y) {};
\foreach \x in {1,2,3}{%
	\foreach \y in {0,1}{%
		\point (b\x\y) at ({\x+.4},{\y-2}) {};}}
\foreach \x/\y in {0/6,1/7,2/8}
	\point (c0\x) at (\y) {};
\foreach \x in {1,2,3}{%
	\foreach \y in {0,1}{%
		\point (c\x\y) at ({\x+.4},{\y+1}) {};}}
\foreach \x in {a,b,c}{%
	\foreach \y/\z in {1/2,2/3}{%
		\foreach \w in {0,1}{%
			\foreach \u in {0,1}{%
				\draw (\x\y\w) -- (\x\z\u);}}}}
\foreach \x in {a,b,c}{%
	\foreach \y in {1,2,3}{%
		\draw (\x\y0) -- (\x\y1);}}
\foreach \x in {a,b,c}{%
	\foreach \y in {0,1,2}{%
		\foreach \z in {0,1}{%
			\draw (\x0\y) -- (\x1\z);}}}
\foreach \x in {0,...,8}
	\draw (\x) \vertex;
\foreach \x in {a,b,c}{%
	\foreach \y in {1,2,3}{%
		\foreach \z in {0,1}{%
			\draw (\x\y\z) \vertex;}}}
\foreach \x/\y/\z in {-3.8/-.5/-1,3.4/-2/1,3.4/1/1}{%
	\foreach \w in {0.2,0.45,0.7}{%
		\fill ({\x+\z*\w},{\y+.5}) \ellipsisdot;}}
\end{tikzpicture}}

\newcommand{\figclam}{%
\begin{tikzpicture}[baseline=0pt]
\foreach \x in {0,1}{
	\foreach \y in {0,1}{
		\path (0,1) ++(120:\x) +(210:\y) \insidepoint (u\x\y) {};
		\path (0,0) ++(-120:\x) +(-210:\y) \insidepoint (v\x\y) {};}}
\foreach \x in {u,v}{
	\foreach \y in {0,1}{
		\draw (\x\y0) -- (\x\y1);
		\foreach \z in {0,1}{
			\draw (\x0\y) -- (\x1\z);}}}
\foreach \x in {0,1}{
	\foreach \y in {0,1}{
		\foreach \z in {0,1}{
			\draw (\x,\y) -- ({\x+1},\z);}}}
\foreach \x in {0,1,2}
	\draw (\x,0) -- (\x,1);
\foreach \x in {0,1,2}{
	\foreach \y in {0,1}{
		\draw (\x,\y) \vertex;}}
\foreach \x in {0,1}{
	\draw (u1\x) \vertex;
	\draw (v1\x) \vertex;}
\draw (u01) \vertex;
\foreach \z in {0.2,0.45,0.7}
	\fill ({2+\z},0.5) \ellipsisdot;
\end{tikzpicture}}

\hypersetup{%
	pdftitle={Some local-global phenomena in locally finite graphs},
	pdfauthor={Armen S. Asratian, Jonas B. Granholm, Nikolay K. Khachatryan},
	pdfkeywords={Hamilton cycle, local conditions, infinite graphs, Hamilton curve}}

\begin{document}

\title{Some local--global phenomena in locally finite graphs}
\author{Armen S. Asratian%
	\footnote{Department of Mathematics, Linköping University, email: armen.asratian@liu.se},
	Jonas B. Granholm%
	\footnote{Department of Mathematics, Linköping University, email: jonas.granholm@liu.se},
	Nikolay K. Khachatryan%
	\footnote{Synopsys Armenia CJSC, email: nikolay@synopsys.com}}
\date{}
\maketitle

\begin{figure}[b!]
\vspace{-1em}
\small
\href{https://doi.org/10.1016/j.dam.2019.12.006}{https://doi.org/10.1016/j.dam.2019.12.006}
\par\medskip
\copyright\ 2019. This manuscript version is made available under the \textsc{cc~by-nc-nd} 4.0 license,
\href{http://creativecommons.org/licenses/by-nc-nd/4.0/}{http://creativecommons.org/licenses/by-nc-nd/4.0/}
\end{figure}

\begin{abstract}
\noindent
In this paper we present some results for a connected infinite graph $G$ with finite degrees
where the properties of balls of small radii guarantee
the existence of some Hamiltonian and connectivity properties of $G$.
(For a vertex $w$ of a graph $G$ the ball of radius $r$ centered at $w$ is the subgraph
of $G$ induced by the set $M_r(w)$ of vertices whose distance from $w$ does not exceed $r$).
In particular, we prove that if every ball of radius~2 in $G$ is 2-connected and $G$
satisfies the condition $d_G(u)+d_G(v)\geq |M_2(w)|-1$ for each path $uwv$ in $G$, where $u$ and $v$ are non-adjacent vertices, then
$G$ has a Hamiltonian curve, introduced by K\"undgen, Li and Thomassen (2017).
Furthermore, we prove that if every ball of radius~1 in $G$ satisfies Ore's condition (1960)
then all balls of any radius in $G$ are Hamiltonian.
\end{abstract}

\bgroup
\noindent
\setlength{\parfillskip}{0pt}%
\par
\egroup

\section{Introduction}

Interconnection between local and global properties of mathematical objects
has always been a subject of investigations in different areas of mathematics.
Usually by local properties of a mathematical object, for example a function,
we mean its properties in balls with small radii. A general question is the following:
How well can global properties of a mathematical object be inferred from the local properties?

If the mathematical object under consideration is a graph,
balls of radius~$r$ are defined only for integers~$r\ge0$.
For a vertex~$u$ of a graph $G$ the \emph{ball of radius $r$ centered at~$u$} is the subgraph
of~$G$ induced by the set $M_r(u)$ of vertices whose distance from~$u$ does
not exceed~$r$.
In the present paper we consider graphs without loops and multiple edges.
The following problem arises naturally:

\begin{problem}
What can we say about global properties of a graph using balls of small radii?
\end{problem}
A number of existing results in graph theory give strong interconnections between local and global properties of a graph.
Consider the following example:

\begin{example}
\label{ex:vizing}
A graph $G$ is $k$-edge colorable if its edges can be colored with
$k$~colors so that no pair of adjacent edges have the same color.
Vizing's theorem~\cite{vizing64} on $k$-edge colorings can be formulated as follows:
A graph $G$ has a $k$-edge coloring if
the degree of every vertex of $G$ is strictly less than $k$.

Thus a local property (every vertex degree is strictly less than $k$) implies that $G$ has a global property
($G$ is $k$-edge colorable).
\end{example}

In contrast with \cref{ex:vizing}, the property of a graph of being connected
cannot be recognized using balls of small radii only, because in any graph (connected or disconnected) all balls of any radius are connected.
However note that any result on a connected graph $G$ concerning
a global property can be reformu\-lated in terms of components of $G$
without mentioning the connectedness of~$G$. Consider an example:

\begin{example}
An Euler tour of a graph $G$ is a walk in $G$ that starts and finishes at the same vertex and
traverses each edge exactly once.
Euler's theorem (see, for example,~\cite{diestel}) says that a connected graph $G$ has an Euler
tour if and only if every vertex of~$G$ has even degree. This theorem can be reformulated
as follows:
Every non-trivial component of a graph~$G$ has an Euler
tour if and only if every vertex of~$G$ has even degree.
\end{example}

Some other global properties of graphs were investigated in~\cite{linial93} by using the properties of balls of small radii.
In this paper we consider mostly Hamiltonian properties of graphs. A finite graph $G$ is called \emph{Hamiltonian} if it has a \emph{Hamilton cycle}, that is, a cycle
containing all the vertices of $G$.
There is a vast literature in graph theory devoted to
obtaining sufficient conditions for Hamiltonicity
(see, for example, the surveys~\cite{gould03,gould14}).

Almost all of the existing sufficient conditions for a finite graph $G$ to
be Hamiltonian
contain some global parameters of $G$ (e.g., the number of vertices)
and only apply to graphs with large edge density
(\,$|E(G)|\geq \text{constant}\cdot|V(G)|^2$\,) and/or small diameter
(\,$o(|V(G)|)$\,).
The following two classical theorems are examples of such results:

\begin{theorem}[Ore~\cite{ore60}]
\label{oldthm:ore}
A finite graph $G=\bigl(V(G),E(G)\bigr)$
with $|V(G)|\geq 3$ is Hamiltonian if $d_G(u)+d_G(v)\geq |V(G)|$
for each pair of non-adjacent vertices~$u$ and $v$ of $G$, where $d_G(u)$ denotes the degree of~$u$.
(A graph satisfying this condition is called an \emph{Ore graph}.)
\end{theorem}

\begin{theorem}[Jung~\cite{jung78}, Nara~\cite{nara80}]
\label{jung}
Let $G=\bigl(V(G),E(G)\bigr)$ be a finite 2-connected graph such that $d_G(u)+d_G(v)\geq |V(G)|-1$
for each pair of non-adjacent vertices~$u,v$. Then either $G$ is Hamiltonian or $G\in \cal K$
where
$\mathcal{K}=\{\,G:K_{p,p+1}\subseteq G\subseteq K_p\vee \overline{K_{p+1}}
\text{ for some }p\geq2\,\}$ ($\vee$ denotes the join operation).
\end{theorem}

Asratian and Khachatryan~\cite{asratyan85,hasratian90,asratian06,asratian07}
showed that many of the global sufficient conditions for Hamiltonicity of a finite graph $G$
have local analogues where
every global parameter of $G$ is replaced by a parameter of a ball with small radius.
Such results are called \emph{localization theorems} and give a possibility
to extend known classes of Hamiltonian graphs.
For example, the following generalization of Ore's theorem was obtained in
\cite{hasratian90} (see also \cite[Thm.~10.1.3]{diestel}):

\begin{theorem}[Asratian and Khachatryan~\cite{hasratian90}]
\label{oldthm:L0}
\label{prop:localoreL0}
Let $G$ be a connected finite graph on
at least 3 vertices where for every vertex $w\in V(G)$ the condition $d_G(u)+d_G(v)\geq |N(u)\cup N(v)\cup N(w)|$
holds for every path $uwv$ with $uv\notin E(G)$, where $N(w)$ denotes the {set of neighbors} of~$w$. Then $G$ is Hamiltonian.
\end{theorem}

A generalization of \cref{jung} was obtained in~\cite{asratian06}:

\begin{theorem}[Asratian~\cite{asratian06}]
\label{localjung}
Let $G$ be a connected finite graph with $|V(G)|\geq 3$ where all balls of radius~2 in $G$ are 2-connected and $d_G(u)+d_G(v)\geq |M_2(w)|-1$
for every path $uwv$ with $uv\notin E(G)$.
Then either $G$ is Hamiltonian or $G\in \cal K$.
\end{theorem}

Note some phenomena related to these results:

1) Although \cref{jung} is a generalization of Ore's theorem,
the localizations of these two theorems (\cref{oldthm:L0} and \cref{localjung})
are incomparable to each other in the sense that neither theorem implies the other.
For example, the graph on the left hand side in \cref{fig:L0andM22ball}
satisfies the condition of \cref{oldthm:L0}
and does not satisfy the condition of \cref{localjung},
and the graph on the right hand side satisfies the condition of \cref{localjung} and does not satisfy the condition of \cref{oldthm:L0}.

\begin{figure}
\centering
\figclawarmfinitefive
\qquad
\figMtwoonlyfinite
\caption{Two graphs showing that \cref{oldthm:L0,localjung} are incomparable.}
\label{fig:L0andM22ball}
\end{figure}

2) All graphs satisfying the conditions of \cref{oldthm:ore} or \cref{jung} have diameter at most two and large edge density.
In contrast with this, \cref{oldthm:L0} and \cref{localjung} apply to infinite classes of finite graphs~$G$ with large diameter ($\geq \text{constant}\cdot|V(G)|$\,)
and small edge density (\,$|E(G)|\leq \text{constant}\cdot|V(G)|$\,).
For example, the graphs in \cref{fig:L0andM22ball} can be extended to graphs with any diameter.

3) The set of Ore graphs and the set of graphs satisfying \cref{prop:localoreL0}
have similar cyclic properties.
For example, every Ore graph $G$ with $V(G)\geq 4$ is pancyclic
(i.e. contains cycles of all length from~3 to $|V(G)|$),
unless $G=K_{n,n}$ for some $n\geq 2$ (see Bondy~\cite{bondy71}).
Moreover each vertex of an Ore graph $G$ with $|V(G)|\geq 4$
lies on a cycle of every length from 4 to $|V(G)|$ (see Cai Xiao-Tao~\cite{cai84}).
Asratian and Sarkisian~\cite{asratian96} showed that
every graph $G$ satisfying the condition of \cref{prop:localoreL0} has the same properties.

Localization theorems were also found
(see \cite{asratian06,asratian18,asratyan85,hasratian90,asratian07})
for results of Dirac~\cite{dirac52}, Bondy~\cite{bondy80}, Nash-Williams~\cite{nash-williams71},
Bauer et~al.~\cite{bauer89},
H\"aggkvist and Nicoghossian~\cite{haggkvist81}, Moon and Moser~\cite{moon63}.
A general method for localization of global criteria for Hamiltonicity of finite graphs was suggested by the authors in~\cite{asratian18}.

A large part of the results of local nature in Hamiltonian graph theory is devoted to \emph{claw-free graphs},
that is, graphs that have no induced subgraph isomorphic to $K_{1,3}$~\cite{faudree97}.
The following well-known result was obtained in~\cite{oberly79}.

\begin{theorem}[Oberly and Sumner~\cite{oberly79}]
\label{oldthm:oberly}
A finite, connected, claw-free graph~$G$ on at least 3 vertices is Hamiltonian
if for each vertex $u$ of $G$ the subgraph induced by the set of neighbors of $u$ is connected.
\end{theorem}

In 2004–2017 some Hamiltonian properties of finite graphs were extended to infinite locally finite graphs,
that is, infinite graphs where all vertices have finite degrees.
There are two important notions for a locally finite graph $G$ related to this topic.
The first one, called a \emph{Hamilton circle} of~$G$,
was introduced by Diestel and K\"uhn~\cite{diestel04a},
and the other one, called a \emph{Hamiltonian curve} of~$G$,
was introduced by K\"undgen, Li and Thomassen~\cite{kundgen17} (see the definitions of these two concepts in \cref{sec:def}).
Some results on the existence of Hamilton circles in infinite locally finite graphs were
obtained in \cite{bruhn08,georgakopoulos09, heuer15, heuer16,hamann16}

The next result on Hamiltonian curves was proved in~\cite{kundgen17}.
\begin{theorem}[\cite{kundgen17}]
\label{thm:kundgen17}
The following are equivalent for a locally finite graph $G$.
\begin{enumerate}
\item For every finite vertex set $S$, $G$ has a cycle containing $S$.
\item $G$ has a Hamiltonian curve.
\end{enumerate}
\end{theorem}

This theorem gives possibilities to extend some results on finite graphs to infinite graphs.
For example, the following result was noted in~\cite{kundgen17}:

\begin{theorem}
\label{localinfinite}
Let $G$ be a connected, locally finite, infinite graph where $d_G(u)+d_G(v)\geq |N(u)\cup N(v)\cup N(w)|$ for each path $uwv$ with $uv\notin E(G)$.
Then $G$ has a Hamiltonian curve.
\end{theorem}

In this paper we present some results for a connected infinite locally finite graph $G$ where the properties of balls of small radii
guarantee the existence of some Hamiltonian and connectivity properties of $G$.
In particular, we prove that if all balls of radius~2 in $G$ are 2-connected and $d_G(u)+d_G(v)\geq |M_2(w)|-1$ for each path $uwv$ with $uv\notin E(G)$, then
$G$ has a Hamiltonian curve.

\cref{thm:kundgen17} implies that a connected infinite locally finite graph~$G$ has a Hamiltonian curve if any ball of any radius in $G$ is Hamiltonian.
In contrast with this we show that the Hamiltonicity of all balls is not sufficient for $G$ to have a Hamilton circle.
We obtain a similar result for finite graphs:
For any integer $d\geq 3$ there exists a connected non-Hamiltonian finite graph of diameter~$d$
where all balls of~$G$, except $G$ itself, are Hamiltonian.
In contrast with this we show that if every ball of radius~1 in a connected locally finite graph~$G$ (finite or infinite) is an Ore graph,
then every ball of any radius in $G$ is Hamiltonian. We also show that
the $k$-connectedness of all balls of radius~$r$ in a locally finite graph~$G$, where $r$ is an integer,
implies the $k$-connectedness of all balls in $G$ with radius bigger than~$r$.
This is a generalization of a result of Chartrand and Pippert~\cite{chartrand74}.
We finish the paper with a conjecture concerning Hamilton circles.

\section{Definitions and notations}
\label{sec:def}

We use~\cite{diestel} for
terminology and notation not defined here and consider graphs without loops and multiple edges only.
A graph $G$ is called \emph{locally finite} if every vertex of $G$ has finite degree.
A graph $G$ is finite or infinite according to the number of vertices in $G$.

Let $V(G)$ and $E(G)$ denote, respectively, the vertex set
and edge set of a graph $G$, and let $d_G(u,v)$ denote the distance
between vertices
$u$ and $v$ in $G$. The greatest distance between any two vertices in $G$ is the \emph{diameter} of $G$.

For each integer $r\geq 0$ and each $u\in V(G)$ we denote by $N_r(u)$ and $M_r(u)$ the set of all $v\in V(G)$ with $d_G(u,v)=r$ and $d_G(u,v)\leq r$, respectively.
The set $N_1(u)$ is usually denoted by $N(u)$.
The subgraph induced by the set $M_r(u)$ is denoted by $G_r(u)$
and called the ball of radius~$r$ centered at~$u$.
In fact, for each vertex $u$ of a connected finite graph~$G$ there is an integer~$r(u)$ such that $G$ is a ball of radius~$r(u)$ centered at~$u$.

Let $G$ be a connected graph and $v$ a vertex in a ball $G_r(u), r\ge1$.
We call $v$ an \emph{interior vertex} of $G_r(u)$ if
$M_1(v)\subseteq M_r(u)$.
Clearly, every vertex in $G_{r-1}(u)$ is interior for $G_r(u)$,
and if $G=G_r(u)$ then all vertices in $G$ are interior vertices of~$G_r(u)$.

Let $C$ be a cycle of a graph. We denote by $\dir C$ the cycle $C$ with a given
orientation, and by $\revdir C$ the cycle $C$ with the reverse orientation. If $u,v\in
V(C)$ then $u\dir Cv$ denote the consecutive vertices of $C$ from $u$ to $v$ in the
direction specified by $\dir C$. The same vertices, in reverse order, are given by
$v\revdir Cu$. We use $u^+$ to denote the successor of $u$ on
$\dir C$ and $u^-$ to denote its predecessor.
Analogous notation is used with respect to paths instead of cycles.

A path containing all vertices of a graph $H$ is called a \emph{Hamilton path} of $H$.

A graph is \emph{$k$-connected} if the removal of fewer than $k$ vertices results in neither a disconnected graph nor the trivial graph consisting of a single vertex.
The greatest integer $k$ such that $G$ is $k$-connected is the \emph{connectivity} $\kappa(G)$ of $G$.

Let $G$ be an infinite locally finite graph.
A one-way infinite path in $G$ is called a \emph{ray}, and a two-way infinite path is called a \emph{double ray}.
Two rays are equivalent if no finite set of vertices
separate them in~$G$, which
means that for every finite vertex set $S\subset V(G)$ both rays
have a tail (subray) in the same component of $G-S$.
This is an equivalence relation whose equivalence classes are called \emph{ends} of $G$. Every end can be viewed as a particular ``point of infinity''.
The \emph{Freudenthal compactification} $|G|$ of $G$ is a topological space obtained by taking $G$,
seen as 1-complex, and adding the ends of $G$ as additional points.
For the precise definition of $|G|$ see~\cite{diestel}.
It should be pointed out that, inside $|G|$, every ray of $G$ converges to the end of $G$ it is contained in.

Extending the notion of cycles,
Diestel and K\"uhn~\cite{diestel04a} defined \emph{circles} in $|G|$ as
the image of a homeomorphism which maps the unit circle $S^1$ in $\mathbb{R}^2$ to $|G|$.
The graph $G$ is called Hamiltonian if there exists
a circle in $|G|$ which contains all vertices and ends of $G$.
Such a circle is called a \emph{Hamilton circle} of $G$.

A \emph{closed curve} in $|G|$ is the image of a continuous map from the unit circle $S^1$ in $\mathbb{R}^2$ to $|G|$.
A closed curve in the Freudenthal compactification $|G|$ is called a \emph{Hamiltonian curve} of $G$
if it meets every vertex of $G$ exactly once (and hence it meets every end at least once).

\section{Locally finite graphs with Hamiltonian balls}

A graph~$G$ is called \emph{locally Hamiltonian (locally traceable)}
if for every vertex~$u$ of~$G$ the subgraph induced by the set $N(u)$ is Hamiltonian (has a Hamilton path).
In fact, a graph~$G$ is locally traceable if and only if every ball of radius~1 in~$G$ is Hamiltonian. Some cyclic properties of
locally Hamiltonian and locally traceable graphs were found in \cite{asratian98,chen13,dewet18,pareek83,vanaardt16}

We call a locally finite graph $G$ \emph{uniformly Hamiltonian} if every ball of finite radius in $G$ is Hamiltonian.
This concept was defined for finite graphs in~\cite{asratian98}, where some classes of uniformly Hamiltonian graphs were found.
Here we prove the following result which was conjectured in~\cite{asratian98}.\medskip

\begin{theorem}
\label{ballhamiltonicity}
Let $G$ be a connected locally finite graph on at least 3 vertices (finite or infinite),
where every ball of radius~1 is an Ore graph.
Then $G$ is uniformly Hamiltonian.
\end{theorem}

\begin{proof}
By the hypothesis of the theorem
$d_{G_1(w)}(u)+d_{G_1(w)}(v)\geq |M_1(w)|$
for each $w\in V(G)$ and each pair of non-adjacent vertices $u,v\in N(w)$. Clearly,
\[d_{G_1(w)}(u)+d_{G_1(w)}(v)
=|M_1(w)\cap N(u)\cap N(v)|+\card[\big]{M_1(w)\cap \bigl(N(u)\cup N(v)\bigr)}.\]
Then Ore's condition for the ball $G_1(w)$ is equivalent to the condition
\begin{equation}
\label{thm:ballhamiltonicity:eq:eqcond}
|M_1(w)\cap N(u)\cap N(v)|\geq \card[\big]{M_1(w)\setminus\bigl(N(u)\cup N(v)\bigr)}
\end{equation}
for each pair of non-adjacent vertices $u,v\in N(w)$.

Suppose that for some integer $r\geq 1$ one of the balls of radius $r$ in $G$, say
$G_r(x)$, is not Hamiltonian. Among all cycles in $G_r(x)$ which contain $x$, let $C$ be one of maximum length.
Consider a vertex $a\in V(G_r(x))\setminus V(C)$. Let $d_G(x,a)=t$ and $P$ be a shortest path between $x$ and $a$.
Then there are two consecutive vertices $w_1$ and $v$ on $P$
such that $w_1\in V(C)$, $d(x,w_1)\leq t-1$ and $v\notin V(C)$. Clearly,
$M_1(w_1)\subseteq V(G_r(x))$, since $t\leq r$.
Let $\dir C$ be the cycle $C$
with a given orientation, and let
$w_1,\dotsc,w_k$
be the vertices of $W=N(v)\cap V(C)$ occurring on $\dir C$ in the order of their
indices.

Set $W^+=\{w_1^+,\dotsc,w_k^+\}$.
Clearly, for each~$i, 1\leq i\leq k$, the vertices $v$ and $w_i^+$ have no common neighbor in
$G_1(w_i)$ outside~$C$,
because if $vz, zw_i^+\in E(G)$ for some $i$ and $z\in M_1(w_i)\setminus V(C)$, the cycle
$w_ivzw_i^+\dir Cw_i$ in $G_r(x)$
contains $x$ and is longer than $C$.
Therefore $N(w_i^+)\cap N(v)\cap M_1(w_i)\subseteq V(C)$, for $i=1,\dotsc,k.$
This implies that $k\geq 2$, since $w_1^+,v\in M_1(w_1)\setminus\bigl(N(w_1^+)\cup N(v)\bigr)$ and
therefore, by \cref{thm:ballhamiltonicity:eq:eqcond},
$|M_1(w_1)\cap N(w_1^+)\cap N(v)|\geq 2$.

Clearly, $w_i\not=w_{i-1}^+$ for $i=1,\dotsc,k-1$, because
if $w_i=w_{i-1}^+$ for some $i$, the cycle $w_{i-1}vw_i\dir Cw_{i-1}$ in $G_r(x)$
contains $x$ and is longer than $C$. Furthermore, $w_i^+w_j^+\notin E(G)$ for $1\leq i<j\leq k$ because otherwise the cycle
$w_{i}vw_j\revdir Cw_{i}^+w_j^+\dir Cw_i$ in $G_r(x)$
contains $x$ and is longer than $C$. Thus, the set $W^+\cup\{v\}$ is an independent set.

Now we count the number of edges
$e(W,W^+)$ between $W$ and $W^+$.
Since $d(v,w_i^+)=2$ for each $i=1,\dotsc,k$, we
obtain from \cref{thm:ballhamiltonicity:eq:eqcond} that
\begin{equation}
\label{thm:ballhamiltonicity:eq:2}
\sum_{i=1}^k|M_1(w_i)\cap N(w_i^+)\cap N(v)|\geq
\sum_{i=1}^k\card[\big]{M_1(w_i)\setminus\bigl(N(w_i^+)\cup N(v)\bigr)}.
\end{equation}
Furthermore
\begin{equation}
\label{thm:ballhamiltonicity:eq:3}
e(W,W^+)\geq \sum_{i=1}^k|M_1(w_i)\cap N(w_i^+)\cap N(v)|
\end{equation}
and
\begin{equation}
\label{thm:ballhamiltonicity:eq:4}
\sum_{i=1}^k\card[\big]{M_1(w_i)\setminus\bigl(N(w_i^+)\cup N(v)\bigr)}\geq
e(W,W^+)+k
\end{equation}
because $M_1(w_i)\cap N(w_i^+)\cap N(v)\subseteq W\cap N(w_i^+)$ and
$v\in M_1(w_i)\setminus\bigl(N(w_i^+)\cup N(v)\bigr)$ for each $i=1,\dotsc,k$.
But \cref{thm:ballhamiltonicity:eq:3} and \cref{thm:ballhamiltonicity:eq:4} contradict \cref{thm:ballhamiltonicity:eq:2}.
Therefore, $C$ contains all vertices of the ball $G_r(x)$,
that is, $G_r(x)$ is Hamiltonian.
\end{proof}

Note that \cref{oldthm:oberly} can be formulated in terms of balls as follows:
A finite connected graph~$G$ on at least 3 vertices is Hamiltonian if for every vertex $u\in V(G)$ the ball
$G_1(u)$ satisfies the condition $\kappa\bigl(G_1(u)\bigr)\geq 2\geq \alpha\bigl(G_1(u)\bigr)$,
where $\kappa\bigl(G_1(u)\bigr)$ is the connectivity of $G_1(u)$ and $\alpha\bigl(G_1(u)\bigr)$ is
the maximum number of mutually non-adjacent vertices of $G_1(u)$.

The next theorem is an extension of \cref{oldthm:oberly} and a result obtained in~\cite{asratian98} for finite claw-free graphs. The proof
for infinite locally finite graphs is the same as in~\cite{asratian98}.

\begin{theorem}
\label{balloberly}
Let $G$ be a connected, locally finite, claw-free graph on at least three vertices (finite or infinite)
where for each vertex~$u$ the subgraph induced by the set of neighbors of $u$ is connected.
Then $G$ is uniformly Hamiltonian.
\end{theorem}

Note that \cref{ballhamiltonicity} and \cref{balloberly} are incomparable in the sense that neither theorem implies the other.
For example, the graph on the right hand side in \cref{fig:localoreandclawfree} is not claw-free and satisfies the condition of \cref{ballhamiltonicity},
and the graph on the left hand side satisfies the condition of
\cref{balloberly}, but does not satisfy the condition of \cref{ballhamiltonicity}.

\begin{figure}
\centering
\figclam
\qquad
\figclawarmore
\caption{Two graphs showing that \cref{ballhamiltonicity,balloberly} are incomparable.}
\label{fig:localoreandclawfree}
\end{figure}

In \cref{ballhamiltonicity,balloberly}, the Hamiltonicity of balls of radius~1 of a connected, finite graph~$G$
induces Hamiltonicity of $G$ and all balls in~$G$.
This is not true in general.
For example, it is known that there exist finite connected non-Hamiltonian graphs where all balls of radius~1 are Hamiltonian
(see, for example,~\cite{pareek83}).
We will prove a stronger result.

\begin{theorem}
\label{thm:finitehamballsnonham}
For any integer $d\ge3$ there exists a non-Hamiltonian finite graph $G$ of diameter $d$ such that
every ball of $G$, except $G$ itself, is Hamiltonian.
\end{theorem}
\begin{proof}
Let $H_1, H_2, H_3,\dotsc$ be a sequence of graphs where
\begin{itemize}
\item
$V(H_1)=\set{a_1,a_2,b_1,b_2,b_3,b_4}$,\\
$E(H_1)=\set{a_1b_1,a_1b_2,a_2b_3,a_2b_4}$,
\item
and, for $i\ge2$,\\
$V(H_i)=\set{b_{4i-7}\dotsc,b_{4i},c_{4i-7},\dotsc,c_{4i-4}}$,\\
$E(H_i)=\set{b_jb_k:4i-7\le j<k\le4i}\cup\set{c_jb_j,c_jb_{j+4}:4i-7\le j\le4i-4}$.
\end{itemize}
Consider a graph
$G(d)=\bigcup_{i=1}^{d-1} H_i\cup F_d$ where $F_d$ is a graph with
\begin{align*}
V(F_d)&=\set{a_3,a_4,b_{4d-7},b_{4d-6},b_{4d-5},b_{4d-4}}\\
E(F_d)&=\set{a_3b_{4d-7},a_3b_{4d-6},a_4b_{4d-5},a_4b_{4d-4}}.
\end{align*}

\begin{figure}
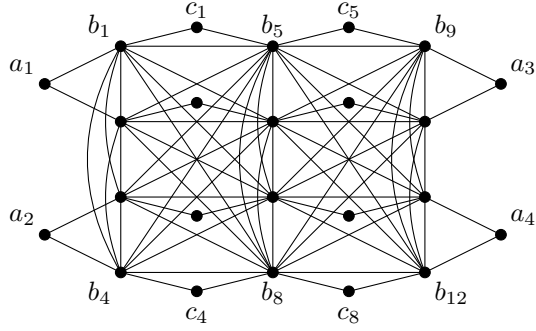

\centering
\figotherhamballsnonham
\caption{The graph $G(4)$.}
\label{fig:hamballsnonham}
\end{figure}

The graph $G(4)$ can be seen in \cref{fig:hamballsnonham}.
Clearly the diameter of~$G(d)$ equals~$d$.
Furthermore, $G(d)$~is not Hamiltonian because the edges incident with vertices of degree~2
induce two disjoint cycles.

Put $G=G(d)$. For any $S\subset V(G)$ we denote by $\induced{S}$ the subgraph of $G$ induced by the vertices in $S$.
Now we will show that for every vertex $v\in V(G)$ the ball $G_r(v)$ is Hamiltonian if $G_r(v)\ne G$.

It is not difficult to verify that any such ball with $r\geq 3$ is isomorphic, for some $s\geq r$, to one of the graphs $\bigcup_{i=1}^s H_i$ and $\bigcup_{i=2}^s H_i$.
Both graphs are Hamiltonian: the first one has a Hamilton cycle $C$ consisting of all edges incident with vertices of degree~2,
and the edges $b_{4s-3}b_{4s}$ and $b_{4s-2}b_{4s-1}$, and the second one has a Hamilton cycle $C'$ obtained from $C$ by
deleting the edges $b_1a_1, a_1b_2, b_3a_2, a_2b_4$ and adding the edges $b_1b_2$ and $b_3b_4$.

The ball $G_1(a_i), i=1,2,3,4,$ and $G_1(c_j), j=1,\dotsc,4d-8,$ are triangles and thus Hamiltonian.
The ball $G_1(b_i)$ is isomorphic to the graph $\indset{b_1,\dotsc,b_{12},c_1,c_5}$ if $5\leq i\leq 4d-8$ and to
$\indset{a_1,b_1,\dotsc,b_8,c_1}$ otherwise. The Hamiltonicity of these graphs is evident.

Finally consider the balls of radius~2 which differ from $G$:
\begin{itemize}
\item
$G_2(a_i)$ is isomorphic to $\indset{a_1,b_1,\dotsc,b_{8},c_1,c_2}$, $i=1,2,3,4$.
\item
$G_2(b_i)$ is isomorphic to $\bigcup_{j=2}^{5} H_j$ if $9\leq i\leq 4d-12$, to $\bigcup_{j=1}^{4} H_j$ if $5\leq i\leq 8$
or $4d-11\leq i\leq 4d-8$, and to $\bigcup_{j=1}^{3} H_j$ if $1\leq i\leq 4$ or $4d-7\leq i\leq 4d-4$.
\item
$G_2(c_i)$ is isomorphic to $\indset{b_1,\dotsc,b_{16},c_1,c_5,c_9}$ if $5\leq i\le4d-12$
and to $\indset{a_1,b_1,\dotsc,b_{12},c_1,c_5}$ if $i\le4$ or $i\ge4d-11$,
unless $d=3$ in which case $G_2(c_i)$ is isomorphic to $\indset{a_1,a_3,b_1,\dotsc,b_8,c_1}$ for $i=1,2,3,4$.
\end{itemize}
It is not difficult to verify that these graphs are Hamiltonian.
Therefore all balls of $G$, except $G$ itself, are Hamiltonian.
\end{proof}

Now we consider the ability of infinite uniformly Hamiltonian graphs
to have Hamilton circles and Hamiltonian curves.

\begin{proposition}
If all balls of all radii in a connected, locally finite, infinite graph~$G$ are Hamiltonian, then $G$ has a Hamiltonian curve.
\end{proposition}
\begin{proof}
Let $S$ be a finite subset of $V(G)$ and let $r$ be the maximum distance between any two vertices in~$S$.
Choose a vertex $a\in S$.
Then all vertices of $S$ are in the ball $G_r(a)$.
Since $G_r(a)$ is Hamiltonian, there is a cycle~$C$ containing~$S$.
Therefore, by \cref{thm:kundgen17}, $G$ has a Hamiltonian curve.
\end{proof}

In contrast with this result we have the following:

\begin{proposition}
There exists an infinite, connected, locally finite graph $H$ such that
every ball of finite radius in $H$ is Hamiltonian,
but $H$ has no Hamilton circle.
\end{proposition}
\begin{proof}

Consider the infinite graph $H=\bigcup_{i=1}^\infty H_i$
where $H_1, H_2, H_3,\dotsc$ is the sequence of graphs defined in the proof of \cref{thm:finitehamballsnonham}.
The graph $H$ can be seen in \cref{fig:infinitehamballsnonham}.
One can prove that every ball of finite radius is Hamiltonian
using the same argument as in the proof of \cref{thm:finitehamballsnonham}.

\begin{figure}
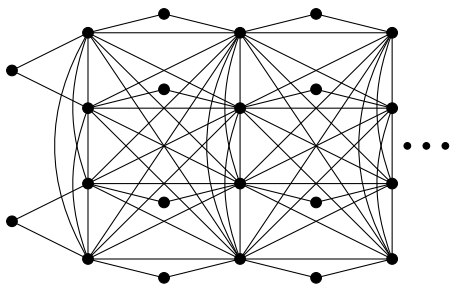

\centering
\figunihamnohamcircleoneend
\caption{An infinite graph $H$ where any ball of of finite radius is Hamiltonian,
but $H$ has no Hamilton circle.}
\label{fig:infinitehamballsnonham}
\end{figure}

Any Hamilton circle in $H$ would need to pass through
all vertices of degree~2, and thus all edges incident with those vertices.
But these edges induce two disjoint double-rays,
which means that the unique end of $H$ would be traversed twice, a contradiction.
Thus $H$ has no Hamilton circle.
\end{proof}

\section{A phenomenon related to ball connectivity}

A graph $G$ is called \emph{locally $k$-connected} if,
for each vertex $u\in V(G)$, the subgraph induced by the set $N(u)$ is $k$-connected.
Chartrand and Pippert~\cite{chartrand74} showed that
if $G$ is a locally $(k-1)$-connected finite graph, $k\geq 2$,
then every component of $G$ is $k$-connected.
This means that if every ball of radius~1 is $k$-connected,
then every component of $G$ is also $k$-connected.
We will prove a stronger result:

\begin{theorem}
\label{thm:k-conn-r-to-r'}
Let $G$ be a locally finite graph and $r$ a positive integer. If all balls of radius $r$ in $G$ are $k$-connected, $k\geq 2$, then all balls in $G$ of radius bigger
than $r$ are $k$-connected.
\end{theorem}

\cref{thm:k-conn-r-to-r'} follows from the following theorem:

\begin{theorem}
\label{thm:ballconnectivity}
Let $G$ be a locally finite graph and $r\geq 1$ be an integer.
If all balls of radius $r$ in $G$ are $k$-connected, $k\geq 2$,
then all balls of radius $r + 1$ in $G$ are $k$-connected too.
\end{theorem}

\begin{proof}
We will consider two cases.
\begin{case}
$r=1$.
\end{case}
It is clear that $k$-connectedness of $G_1(x)$ implies that
$|M_1(x)|\geq k+1$ for each $x\in V(G)$. Suppose that for a
vertex $x$ the ball $G_2(x)$ is not $k$-connected. Then there exists
a subset $S\subset M_2(x)$ such that $|S|<k$ and $G_2(x)-S$ is disconnected.
Among all such sets $S$, let $S_0$ be one of minimum cardinality. Clearly, $S_0$ contains
an interior vertex $v$ of $G_2(x)$ because otherwise $G_2(x)-S_0$ is connected. Thus $M_1(v)\subset M_2(x)$.
The minimality
of $S_0$ implies
that there are two neighbors $w_1$ and $w_2$ of $v$ such that
$w_1$ and $w_2$ belong to different components in $G_2(x)-S_0$. Then the set $S_0\cap M_1(v)$
separates $w_1$ and $w_2$ in $G_1(v)$ and $|S_0\cap M_1(v)|<k$. This contradicts the
$k$-connectedness of $G_1(v)$. Hence all balls of radius~2 in $G$ are $k$-connected.

\begin{case}
$r\ge2$.
\end{case}
We will show that the graph $G_{r+1}(x)-S$ is connected
for each vertex $x$ and each subset $S\subset M_{r+1}(x)$ with $|S|<k$.
Since $G_r(x)$ is $k$-connected, the graph $G_{r}(x)-S$ is connected.
Furthermore, $N(x)\setminus S\not=\emptyset$,
since otherwise the set $N(x)$ is a cutset of size strictly less than $k$ in $G_r(x)$.
Choose a vertex $y\in N(x)\setminus S$.
Then $y\in M_r(u)\setminus S$ for each $u\in M_1(x)$, since $r\geq 2$ and $y\notin S$.
Now for each vertex $z\in M_{r+1}(x)\setminus S$
there is a vertex $u\in M_1(x)$ such that $z\in M_r(u)\setminus S$, because
\begin{equation*}
M_{r+1}(x)=\bigcup_{\mathclap{u\in M_1(x)}}M_r(u).
\end{equation*}
Since $G_r(u)$ is $k$-connected and $|S|<k$,
there is a path between $y$ and $z$ in $G_r(u)-S$ and, therefore, in $G_{r+1}(x)-S$ too.
Thus $G_{r+1}(x)-S$ is connected, so we can conclude that $G_{r+1}(x)$ is $k$-connected.
\end{proof}

\begin{corollary}
If every ball of radius~1 in a connected graph $G$ is $k$-connected, $k\geq 2$,
then $G$ and all balls of any radius in $G$ are $k$-connected.
\end{corollary}

\section{Two classes of infinite graphs with Hamiltonian curves}
\label{sec:infinite}

The proofs of many local sufficient conditions for the existence of a Hamilton cycle
in a finite graph work by starting with an arbitrary cycle and iteratively extending it until it covers all vertices of the graph.
If the extensions of the cycles are chosen carefully enough,
such a proof can be used to prove the existence of Hamiltonian curves in infinite locally finite graphs by applying \cref{thm:kundgen17}.
In this section we give two examples of such an approach.

Our first result concerns the following theorem:

\begin{theorem}[Chvátal and Erdős~\cite{chvatal72}]
A finite graph $G$ on at least 3 vertices is Hamiltonian if $\kappa (G)\geq \alpha (G)$, where
$\kappa (G)$ is the connectivity of $G$ and $\alpha(G)$ is the maximum number of mutually non-adjacent vertices of $G$.
\end{theorem}

Khachatryan~\cite{khachatrian85} noted that the proof of this theorem
given in~\cite{chvatal72}, can be used to prove the following result:

\begin{theorem}[Khachatryan~\cite{khachatrian85}]
Let $r$ be a positive integer and $G$ a connected finite graph on at least three vertices, where $\kappa(G_r(u))\geq \alpha(G_{r+1}(u))$
for every $u\in V(G)$. Then $G$ is Hamiltonian.
\end{theorem}

We extend this result to infinite graphs by slightly changing the proof in~\cite{khachatrian85}.

\begin{theorem}
\label{thm:infinitechvatal}
Let $r$ be a positive integer and $G$ a connected, infinite, locally finite graph
such that $\kappa(G_r(u))\geq \alpha(G_{r+1}(u))$ for each vertex $u$ of $G$. Then $G$ has a Hamiltonian curve.
\end{theorem}

\begin{proof}
We will show that for any finite vertex set $S\subset V(G)$ there is a cycle in $G$ containing $S$.
Let $q$ be the maximum distance between any two vertices in $S$. Choose a vertex $a\in S$ and an integer $n=q+r+1$.
Then the ball $G_n(a)$ contains the set $S$ and, moreover, for every $u\in S$ the vertices of $G_r(u)$ are interior vertices of the ball $G_n(a)$.
Among all cycles in $G_n(a)$ which contain $a$,
let $C$ be one of maximum length. Suppose to the contrary that $S\setminus V(C)\not=\emptyset$. Consider a vertex $y\in S\setminus V(C)$ and a shortest
$(a,y)$-path in $G_n(a)$. Clearly, there are two adjacent vertices $v$ and $u$ on this path
such that $v\notin V(C)$, $u\in V(C)$ and $u$ is an interior vertex of $G_q(a)$.
Let $\dir C$ be the cycle $C$ with a given orientation.
We have that $2\leq \alpha(G_{r+1}(u))\leq \kappa(G_r(u))$ since $vu^+\notin E(G)$. Then, by Menger's
theorem~\cite{diestel}, in $G_r(u)$
there are $k$ internally disjoint $(v,u^+)$-paths $Q_1,\dotsc,Q_k$, where $k=\kappa(G_r(u))$.
Maximality of $C$ implies that each $Q_i$ has at least one common vertex with $C$. This means
that there are paths $P_1,\dotsc,P_k$ having initial vertex $v$ that are pairwise disjoint,
apart from $v$, and that share with $C$ only their terminal vertices $v_1,\dotsc,v_k$,
respectively. Furthermore, maximality of $C$ implies that $vv_i^+\notin E(G)$ for each $i=1,\dotsc,k$. Then
there is a pair $i,j$ such that $1\leq i<j\leq k$ and $v_i^+v_j^+\in E(G)$,
as otherwise there are $k+1$ mutually non-adjacent vertices $v,v_1^+,\dotsc,v_k^+$ in $G_{r+1}(v)$ which
contradicts the condition $\alpha(G_{r+1}(v))\leq \kappa(G_r(v))$.
Since the vertices of $G_r(u)$ are interior vertices of the ball $G_n(a)$, the paths $P_1,\dotsc P_k$ lie in $G_n(a)$.

Now by deleting the edges $v_iv_i^+$ and $v_jv_j^+$ from $C$ and adding the edge $v_i^+v_j^+$
together with the paths $P_i$ and $P_j$, we obtain in $G_n(a)$ a cycle that is longer than $C$
and contains $a$;
a contradiction. Therefore, $C$ contains the set $S$. Then, by \cref{thm:kundgen17}, $G$ has a Hamiltonian curve.
\end{proof}

Note that for any integer $r\geq 1$
there is an infinite locally finite graph that satisfies the conditions of \cref{thm:infinitechvatal}.
Consider, for example, the graph $G=G(r)$ which is defined as follows:
its vertex set is $\cup_{i=0}^{\infty}V_i$,
where $V_0,V_1, V_2, \dotsc$ are pairwise disjoint finite sets with $|V_0|=2$ and $|V_i|\geq r+3$ for each $i\geq 1$,
and two vertices $x,y$ of $G$ are adjacent if and only if either
$x,y\in V_i$ for some $i\geq 1$, or $x\in V_i$, $y\in V_{i+1}$
for some $i\geq 0$.
It is not difficult to verify that the graph $G=G(r)$ is not claw-free and satisfies the conditions of \cref{thm:infinitechvatal}.

Our second result is an extension of \cref{localjung}.

\begin{theorem}
\label{localjung1}
Let $G$ be a connected, infinite, locally finite graph where
every ball of radius~2 is 2-connected and $d_G(u)+d_G(v)\geq |M_2(w)|-1$ for each path $uwv$ with $uv\notin E(G)$.
Then $G$ has a Hamiltonian curve.
\end{theorem}

Note the following
two simple properties of a graph $G$ satisfying the condition of \cref{localjung1}:

\begin{property}
\label{thm:localjung1:property1}
If $uwv$ be a path in $G$ with $uv\notin E(G)$, then the condition $d_G(u)+d_G(v)\geq |M_2(w)|-1$ is equivalent to the condition
$|N(u)\cap N(v)|\geq |M_2(w)\setminus (N(u)\cup N(v))|-1$.
\end{property}

\begin{property}
\label{thm:localjung1:property2}
Let $G$ be a graph satisfying the conditions of \cref{localjung1},
$\dir C$ a cycle in $G$ with a given orientation,
and let $v, w$ be vertices such that $v\in V(G)\setminus V(\dir C)$,
$w\in N(v)\cap V(\dir C)$,
$N(v)\cap N(w^+)=N(v)\cap N(w^-)=\{w\}$, and
$w^+, w^-\notin N(v)$.
Then $w^+$ is adjacent to every vertex in $M_2(w)\setminus (M_1(v)\cup \{ w^+\})$
and $w^-$ is adjacent to every vertex in $M_2(w)\setminus (M_1(v)\cup \{ w^-\})$.
In particular, $w^+w^-\in E(G)$.
\end{property}

\begin{boldproof}[Proof of \cref{localjung1}]
Let $G$ be a graph satisfying the conditions of \cref{localjung1},
and let $q$ be the maximum distance between any two vertices in $S$.
Choose a vertex $a\in S$ and let $n=q+5$.
Then the ball $G_n(a)$ contains the set $S$ and, moreover, $M_2(u)\subset M_n(a)$ for every $u\in S$.
Among all cycles in $G_n(a)$ which contain the vertex $a$, let $C$ be one of maximum length. We will show that $C$ contains $S$.
Suppose to the contrary
that $S\setminus V(C)\not=\emptyset$. We will show that this leads to a contradiction, by constructing in $G_n(a)$ a longer cycle $C'$ containing the vertex $a$.
Let $\dir C$ be the cycle $C$ with a given orientation.

Further we follow the proof of \cref{localjung} given in~\cite{asratian06} with some changes. The main changes are that all transformations of $C$ will be done
inside the ball $G_n(a)$, and therefore we need more technical details on the structures of the considered balls.

\begin{claim}
\label{thm:localjung1:claim:first}
There is a vertex $v$ in $G_n(a)-V(C)$ such that $M_2(v)\subset M_n(a)$ and $v$ has at least two neighbors on $C$.
\end{claim}
\begin{proof}
Assume that the claim is false, that is,
any vertex $x$ in $G_n(a)-V(C)$ with $M_2(x)\subset M_n(a)$ has at most one neighbor on $C$.

Consider a vertex $y\in S\setminus V(C)$ and a shortest
$(a,y)$-path in $G_n(a)$.
Clearly, there are two adjacent vertices $v$ and $w$ on this path such that $v\notin V(C)$, $w\in V(C)$.
Since $d_G(a,v)\leq d_G(a,y)\leq q=n-5$,
the vertex $v$ is an interior vertex of $G_n(a)$ and $M_5(v)\subset M_n(a)$.

Let $N(v)\cap V(C)=\{w\}$. Then, by \cref{thm:localjung1:property2}, $w^-w^+\in E(G)$.
Since the ball $G_2(w)$ is 2-connected,
there is a path $P=z_1z_2\dotsp z_p$ in $G_2(w)-w$ with $z_1=w^+$ and $z_p=v$.
Clearly, $P$ has an internal vertex which lies on $\dir C$.
Let $z_k$ be a vertex on $V(P)\cap V(C)$, $k>1$, such that all other vertices on the path $z_k\dir Pv$ do not belong to $C$.
Clearly, $k<p-1$ because $N(v)\cap V(C)=\{w\}$.
Furthermore, $M_2(z_{k+1})\subset M_n(a)$ since $z_{k+1}\in M_3(v)$ and $M_5(v)\subset M_n(a)$.
Then, by our assumption, $N(z_{k+1})\cap V(C)=\{z_k\}$ which implies,
by \cref{thm:localjung1:property2}, that $z_k^-z_k^+\in E(G)$.
Clearly, $z_k\not=w^{++}$ since otherwise $G_n(a)$ contains a longer cycle
$C'=w^-w^+wv\revdir Pz_k\dir C w^-$ with $a\in V(C')$.
Thus $w^-w^+, z_k^-z_k^+\in E(G)$, $z_k\not=w^{++}$
and $z_k\in M_2(w)\setminus (M_1(v)\cup \{w^-\})$.
Then, by \cref{thm:localjung1:property2}, $w^-z_k\in E(G)$ which implies that $G_n(a)$ contains a longer cycle $C'=w^-z_k\dir Pvw\dir C z_k^-z_k^+\dir C w^-$ with $a\in V(C')$; a contradiction.
\end{proof}

\begin{claim}
\label{thm:localjung1:claim:last}
There is an interior vertex $u$ in $G_n(a)$ outside of $C$ and a vertex $w\in N(u)\cap V(C)$ such that either $|N(u)\cap N(w^+)|\geq 2$
or $|N(u)\cap N(w^-)|\geq 2$.
\end{claim}
\begin{proof} The proof is by contradiction. Suppose that
\begin{equation}
N(u)\cap N(w^+)=N(u)\cap N(w^-)=\{w\}, \label{E1}
\end{equation}
for each pair $u,w$ where $u$ is an interior vertex in $G_n(a)$ outside of~$C$ and $w\in N(u)\cap V(C)$.

Choose an interior vertex $v$ in $G_n(a)$ outside of~$C$ such that
$M_2(v)\subset M_n(a)$ and $v$ has at least two neighbors on $C$
(such a vertex exists by \cref{thm:localjung1:claim:first}).
Set $W=N(v)\cap V(C)$ and $k=|W|$. Let $w_1,\ldots,w_k$ be the vertices of $W$, occurring on
$\dir C$ in the order of their indices. By \cref{E1}, we have

\begin{equation}
N(v)\cap N(w^+_i)=N(v)\cap N(w^-_i)=\{w_i\}, \qquad (\,i=1,\ldots,k\,).
\label{E2}
\end{equation}

By the condition of \cref{localjung1}, the ball $G_2(v)$ is a 2-connected graph. Consider in $G_2(v)-w_1$ a shortest
path $P=u_0u_1\dotsp u_t$ where $u_0=w_1^+$ and $u_t\in \{w_2,\dotsc,w_k\}$.
By \cref{E2}, $u_1\notin N(v)$.
Then $d(v,u_1)=2$ and there is a vertex $v_1\in N(v)$ which is adjacent to $u_1$. We will show that $v_1=w_1$.
Since $M_2(v)\subset M_n(a)$, $v_1$ is an interior vertex in $G_n(a)$.
Clearly, $v_1\in N(v)\cap V(C)=\{w_1,\dotsc,w_k\}$ because otherwise $G_n(a)$ contains a cycle $C'$ with $a\in V(C')$ which is longer
than $C$. (For example, $C'=w_1vv_1u_1w_1^+\dir C w_1$ if $u_1\notin V(C)$,
and $C'= w_1vv_1u_1w_1^+\dir C u_1^-u_1^+\dir C w_1$ if $u_1\in V(C)$, where $u_1^-u_1^+\in E(G)$ by \cref{thm:localjung1:property2}.)
Suppose that $v_1\in\{ w_2,\dotsc,w_k\}$, say $v_1=w_2$.
Then, $w_1^+\in M_2(w_2)-(M_1(v)\cup \{w_2^+\})$.
Therefore, by \cref{E2} and \cref{thm:localjung1:property2}, $w_2^+$ is adjacent to $w_1^+$. But then the cycle $w_1vw_2\revdir C w_1^+w_2^+\dir C w_1$
is longer than $C$, a contradiction.
Therefore, $v_1\notin\{ w_2,\dotsc,w_k\}$ and $v_1=w_1$,
that is,
\begin{equation}
u_1w_1\in E(G),\quad u_1w_i\notin E(G), \qquad (\,i=2,\ldots,k\,).
\label{E3}
\end{equation}

Since $u_1$ is adjacent to the consecutive vertices $w_1$ and $w_1^+$ on $\dir C$, and $C$ is a longest cycle of $G$,

\begin{equation}
u_1\in V(C).
\label{E4}
\end{equation}

We will now show that
\begin{equation}
vu_2\notin E(G).
\label{E5}
\end{equation}
Suppose to the contrary that $vu_2\in E(G)$.
Then \cref{E3} and $w_1\notin V(P)$ imply that
$u_2\in N(v)\setminus V(C)$.
Also, since $M_2(v)\subset M_n(a)$, $u_1$ is an interior vertex in $G_n(a)$.
We have $u_1u_2\in E(G), u_1\in V(C)$ and $u_2\in N(v)\setminus V(C)$.
Therefore, by \cref{E1} and \cref{thm:localjung1:property2}, $u_1^-u_1^+\in E(G)$.
But then the cycle $w_1vu_2u_1w_1^+\dir Cu_1^-u_1^+\dir Cw_1$ is longer than $C$.

Thus $vu_2\notin E(G)$.
Clearly, $u_2\in M_2(w_1)$ because $w_1u_1, u_1u_2\in E(G)$.
Then by \cref{E5} and \cref{thm:localjung1:property2}, $w_1^+u_2\in E(G)$.
But this contradicts the assumption that
$u_0u_1\dotsp u_t$ is a shortest path with origin $w_1^+$ and terminus in $\{w_2,\dotsc,w_k\}$.

The proof of \cref{thm:localjung1:claim:last} is completed.
\end{proof}

We continue to prove the theorem.
By \cref{thm:localjung1:claim:last}, there is an interior vertex $v\in V(G)$ outside of $C$ and
$w_1\in V(C)$ such that either $|N(v)\cap N(w^+_1)|\geq 2$
or $|N(v)\cap N(w^-_1)|\geq 2$.
Without loss of generality we assume that $|N(v)\cap N(w_1^+)|\geq 2$.
The choice of $C$ implies that $N(v)\cap N(w_1^+)\subseteq V(C)$.
Set $W=N(v)\cap V(C)$ and $k=|W|$. Let $w_1,\ldots,w_k$ be the vertices of $W$, occurring on
$\dir C$ in the order of their indices, $k\geq 2$.
Set $W^+=\{w_1^+,\dotsc,w_k^+\}$.

We will count the number of edges $e(W^+,W)$ between $W^+$ and $W$.
The choice of~$C$ implies that $W^+\cup\{v\}$ is an
independent set, and $N(w_i^+)\cap N(v)\cap (V(G)\setminus V(C))=\emptyset$, for $1\leq i\leq k.$ Moreover, for each path $vw_iw_i^+$, $ 1\leq i\leq k$,
we have, by the hypothesis of the theorem and by \cref{thm:localjung1:property1},
\begin{equation}
|N(v)\cap N(w_i^+)|\geq |M_2(w_i)\setminus (N(v)\cup N(w_i^+))|-1.
\label{E7}
\end{equation}
Obviously,
$N(w_i)\cap W^+\subseteq N(w_i)\setminus (N(v)\cup N(w_i^+)\cup \{v\}).$
Thus,
\[|N(w_i)\cap W^+|\leq |N(w_i)\setminus (N(v)\cup N(w_i^+))|-1
	\leq |M_2(w_i)\setminus(N(v)\cup N(w_i^+))|-1.\]
This and \cref{E7} imply that
$|N(w_i)\cap W^+|\leq |N(v)\cap N(w_i^+)|.$
Hence,
\[e(W^+, W)=\sum_{i=1}^k|N(w_i)\cap W^+|\leq \sum_{i=1}^k|N(v)\cap N(w_i^+)|=e(W^+,W).\]
It follows, for each $i, 1\leq i\leq k$, that
\begin{equation}
N(w_i)\setminus (N(v)\cup N(w_i^+)\cup \{v\})=N(w_i)\cap W^+\subseteq W^+.
\label{E8}
\end{equation}

Noting that $k\geq2$ and the fact that
$|N(w_1^+)\cap N(v)|\geq2$, we now prove by contradiction that
$w_i^+\:=\:w^-_{i+1}$ for each $i=1,\dotsc,k$. (We consider $w_{k+1}=w_1$.)

Assume without loss of generality that $w_1^+\neq w_2^-$, whence
$w_2^-\notin W^+$. Observe that $w_2^-\in N(w_2^+)$, otherwise from \cref{E8}, $w_2^-\in W^+$.
Since~$C$ is a longest cycle, $w^-_2w^-_3 \notin E(G)$.
Hence $w_2^+\neq w_3^-$.
Repetition of this argument shows that $w_i^+\neq w^-_{i+1}$ and
$w_i^+w_i^-\in E(G)$ for all $i\in\{\,1,\ldots,k\,\}$.
By assumption, $N(w_1^+)\cap N(v)$ contains a vertex $x\neq w_1$.
Since $C$ is a longest cycle, $x\in V(C)$, say that $x=w_i$.
But then the cycle $w_1vw_iw_1^+\dir C w_i^-w_i^+\dir C w_1$ is longer than~$C$.
This contradiction proves~that $w_i^+=w^-_{i+1}$ for each $i=1,\dotsc,k$, where $w_{k+1}=w_1$.
Thus $V(C)=W\cup W^+$,
that is, $v$ is adjacent to each second vertex of $V(C)$.

Since $G$ is infinite, $V(G)\setminus (V(C)\cup \{v\})\not=\emptyset$. Furthermore, since all balls of radius~2 in $G$ are 2-connected,
the graph $G$ is 2-connected itself.
Thus there is a vertex $v_1$ outside $C$ that has a neighbor on $C$.

\begin{case}
$N(v_1)\cap W\not=\emptyset$.
\end{case}
Then, by \cref{E8}, $vv_1\in E(G)$.
This and $a\in V(C)$ imply that $d(a,v_1)\leq 3$ and $M_2(v_1)\subset M_n(a)$.
Without loss of generality we assume that $v_1w_1\in E(G)$. Then either $|N(v_1)\cap N(w_1^-)|\geq 2$ or $|N(v_1)\cap N(w_1^+)|\geq 2$.
(Otherwise, by \cref{thm:localjung1:property2}, $w_1^-w_1^+\in E(G)$ and $G_n(a)$ contains a longer cycle $w_1^-w_1^+w_1v_1vw_2\dir Cw_1^-$.)
Using for $v_1$ the same argument as for $v$, we can conclude that $v_1$ is adjacent to each second vertex of $V(C)$, that is, $v_1w_i\in E(G)$,
$ i=1, \dotsc ,k$. Clearly, $w_t^+\not=a$ for some $t, 1\leq t\leq k$. But then there is a longer cycle $w_tvv_1w_{t+1}\dir Cw_t$ containing $a$, a contradiction.

\begin{case}
$N(v_1)\cap W^+\not=\emptyset$ and $N(u)\cap W=\emptyset$ for every $u\in V(G)\setminus (V(C)\cup \{v\})$.
\end{case}
Without loss of generality we assume that $v_1w_1^+\in E(G)$.
Then $M_2(w_1^+)\subset M_n(a)$ and $v_1\in M_n(a)$ because
$d(a,w_1^+)\leq 4$ and $n=q+5\geq 6$.
Since $G_2(w_1^+)$ is 2-connected, there is a $(v_1,w_1)$-path $Q$ in $G_2(w_1^+)-w_1^+$.
Let $z$ be a vertex on $V(Q)\cap V(C)$ such that all other vertices on the path $v_1\dir Qz$ do not belong to $C$.
Clearly, $z=w_j^+$ for some $2\leq j\leq k$ because, by our assumption, $N(u)\cap W=\emptyset$ for every $u\in V(G)\setminus (V(C)\cup \{v\})$.
Then
$G_n(a)$ contains a longer cycle $C'=w_1vw_j\revdir Cw_1^+v_1\dir Qw_j^+\dir Cw_1$ with $a\in V(C')$.

\medskip
This final contradiction shows that $S\subseteq V(C)$. Thus for any finite set $S\subset V(G)$ there is a cycle in $G$ containing $C$. Then, by \cref{thm:kundgen17},
$G$ has a Hamiltonian curve.
The proof of \cref{localjung1} is completed.
\end{boldproof}

The class of graphs satisfying the conditions of \cref{localjung1}
contains some claw-free graphs
(for example, the graph at the top of \cref{fig:M22ballindependent}),
as well as graphs that are not claw-free and
do not satisfy the conditions of \cref{localinfinite}
(for example, the graph at the bottom of \cref{fig:M22ballindependent}).

\begin{figure}
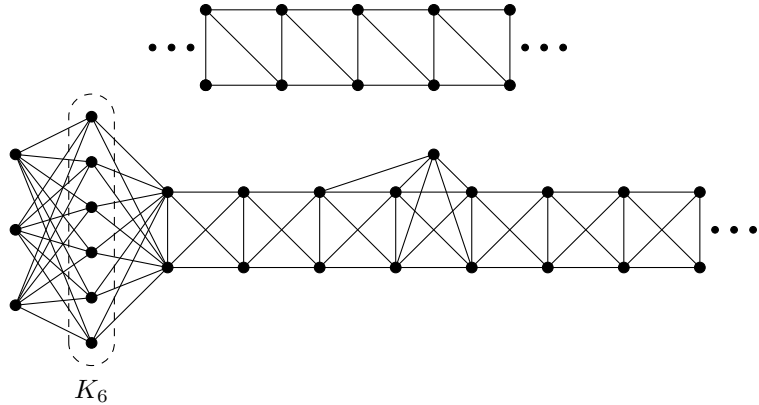

\centering
\figdoubleraysquare
\par
\figMtwoonlyinfinite
\caption{Two graphs satisfying the conditions of \cref{localjung1}.}
\label{fig:M22ballindependent}
\end{figure}

\Cref{localjung1} implies the following result:

\begin{corollary}
\label{localnash}
A connected, infinite, $k$-regular graph $G$ has a Hamiltonian curve if
every ball of radius~2 in $G$ is 2-connected and
$2k\geq |M_2(w)|-1$ for every vertex $w\in V(G)$.
\end{corollary}

Note that the graph at the top of \cref{fig:M22ballindependent}
is an infinite 4-regular graph satisfying the conditions of \cref{localnash}.

\cref{localnash} is an extension
of the following theorem of Nash-Williams~\cite{nash-williams71}:
A 2-connected finite $k$-regular graph $G$ is Hamiltonian if $2k\geq |V(G)|-1$.

The condition $2k\geq |M_2(x)|-1$ in
\cref{localnash}
can be rewritten as
\[2k\geq |M_2(x)|-1=(1+k+|N_2(x)|)-1,\]
which is equivalent to $|N_2(x)|\leq k$, where $N_2(x)$ denotes the set of vertices at distance~2 from $x$.
Therefore \cref{localnash} can be reformulated as follows:

\begin{corollary}
\label{prop:localoreM2regular}
A connected, infinite, $k$-regular graph~$G$ has a Hamiltonian curve if
every ball of radius~2 in $G$ is 2-connected and
the number of vertices at distance~$2$ from any vertex in $G$ is at most~$k$.
\end{corollary}

Diestel~\cite{diestel10} conjectured that the condition of Asratian and Khachatryan for finite graphs (see \cref{oldthm:L0})
guarantees the existence of Hamilton circles
in an infinite locally finite graph $G$:

\begin{conjecture}[Diestel~\cite{diestel10}]
\label{conj:L0circle}
A connected, infinite, locally finite graph $G$ has a Hamilton circle if $d_G(u)+d_G(v)\geq |N(u)\cup N(v)\cup N(w)|$ for each path $uwv$ with $uv\notin E(G)$.
\end{conjecture}

We believe that the following conjecture is true:

\begin{conjecture}
\label{conj:M22ballcircle}
A connected, infinite, locally finite graph $G$ has a Hamilton circle if all balls of radius~2 in $G$ are 2-connected and $d_G(u)+d_G(v)\geq |M_2(w)|-1$ for each
path $uwv$ with $uv\notin E(G)$.
\end{conjecture}

Finally, note that for each integer $n\ge1$
there are infinite locally finite graphs with $n$~ends
satisfying the conditions of these conjectures.
A graph with three ends satisfying the conditions of both conjectures
can be seen in \cref{fig:graphsatisfyingconjectures}.

\begin{figure}
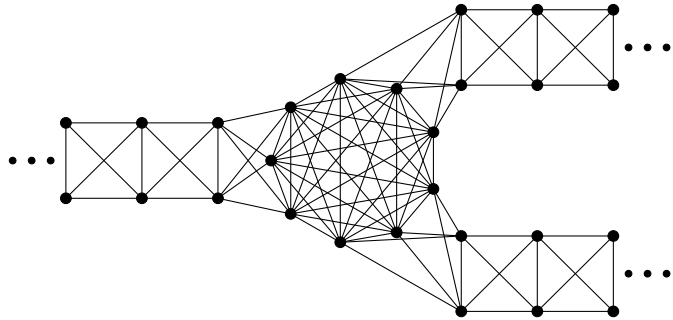

\centering
\figultrametaclaw
\caption{A graph with three ends satisfying the conditions of
\cref{conj:L0circle,conj:M22ballcircle}.}
\label{fig:graphsatisfyingconjectures}
\end{figure}

\section*{Acknowledgment}
The authors thank Carl Johan Casselgren for helpful suggestions on this manuscript.
We also thank the referees for useful remarks.


\begin{thebibliography}{10}

\bibitem{asratian06}
A.~S. Asratian.
\newblock New local conditions for a graph to be {Hamiltonian}.
\newblock {\em Graphs and Combinatorics}, 22(2):153--160, 2006.

\bibitem{asratian18}
A.~S. Asratian, J.~B. Granholm, and N.~K. Khachatryan.
\newblock A localization method in {Hamiltonian} graph theory.
\newblock {\em \href{https://arxiv.org/abs/1810.10430}{\rm arXiv:1810.10430
  [math.CO]}}, 2018.

\bibitem{asratyan85}
A.~S. Asratian and N.~K. Khachatryan.
\newblock Investigation of the {Hamiltonian} property of a graph using
  neighborhoods of vertices ({Russian}).
\newblock {\em Akademiya Nauk Armyansko\u\i\ SSR. Doklady}, 81(3):103--106,
  1985.

\bibitem{hasratian90}
A.~S. Asratian and N.~K. Khachatryan.
\newblock Some localization theorems on {Hamiltonian} circuits.
\newblock {\em Journal of Combinatorial Theory, Series~B}, 49(2):287--294,
  1990.

\bibitem{asratian07}
A.~S. Asratian and N.~K. Khachatryan.
\newblock On the local nature of some classical theorems on {Hamilton} cycles.
\newblock {\em Australasian Journal of Combinatorics}, 38:77--86, 2007.

\bibitem{asratian98}
A.~S. Asratian and N.~Oksimets.
\newblock Graphs with {Hamiltonian} balls.
\newblock {\em Australasian Journal of Combinatorics}, 17:185--198, 1998.

\bibitem{asratian96}
A.~S. Asratian and G.~V. Sarkisian.
\newblock Some panconnected and pancyclic properties of graphs with a local
  {Ore}-type condition.
\newblock {\em Graphs and Combinatorics}, 12(3):209--219, 1996.

\bibitem{bauer89}
D.~Bauer, H.~J. Broersma, H.~J. Veldman, and L.~Rao.
\newblock A generalization of a result of {H{\"a}ggkvist} and {Nicoghossian}.
\newblock {\em Journal of Combinatorial Theory, Series B}, 47(2):237--243,
  1989.

\bibitem{bondy80}
J.~A. Bondy.
\newblock Longest paths and cycles in graphs of high degree.
\newblock University of Waterloo Preprint CORR 80-16.

\bibitem{bondy71}
J.~A. Bondy.
\newblock Pancyclic graphs {I}.
\newblock {\em Journal of Combinatorial Theory, Series~B}, 11(1):80--84, 1971.

\bibitem{bruhn08}
H.~Bruhn and X.~Yu.
\newblock {Hamilton} cycles in planar locally finite graphs.
\newblock {\em SIAM Journal on Discrete Mathematics}, 22(4):1381--1392, 2008.

\bibitem{cai84}
X.-T. Cai.
\newblock On the panconnectivity of {O}re graphs.
\newblock {\em Scientia Sinica, Series~A}, 27(7):684--694, 1984.

\bibitem{chartrand74}
G.~Chartrand and R.~E. Pippert.
\newblock Locally connected graphs.
\newblock {\em Časopis pro pěstování matematiky}, 99(2):158--163, 1974.

\bibitem{chen13}
G.~Chen, A.~Saito, and S.~Shan.
\newblock The existence of a 2-factor in a graph satisfying the local
  {Chv\'{a}tal--Erd\H{o}s} condition.
\newblock {\em SIAM Journal on Discrete Mathematics}, 27(4):1788--1799, 2013.

\bibitem{chvatal72}
V.~Chv\'{a}tal and P.~Erd\"{o}s.
\newblock A note on {Hamiltonian} circuits.
\newblock {\em Discrete Mathematics}, 2(2):111--113, 1972.

\bibitem{dewet18}
J.~P. de~Wet, M.~Frick, and S.~A. van Aardt.
\newblock {Hamiltonicity} of locally {Hamiltonian} and locally traceable
  graphs.
\newblock {\em Discrete Applied Mathematics}, 236:137--152, 2018.

\bibitem{diestel}
R.~Diestel.
\newblock {\em Graph Theory}.
\newblock Springer, 4th edition, 2010.

\bibitem{diestel10}
R.~Diestel.
\newblock Locally finite graphs with ends: {A} topological approach, {II}.
  {Applications}.
\newblock {\em Discrete Mathematics}, 310(20):2750--2765, 2010.

\bibitem{diestel04a}
R.~Diestel and D.~K{\"u}hn.
\newblock On infinite cycles {I}.
\newblock {\em Combinatorica}, 24(1):69--89, 2004.

\bibitem{dirac52}
G.~A. Dirac.
\newblock Some theorems on abstract graphs.
\newblock {\em Proceedings of the London Mathematical Society}, s3-2(1):69--81,
  1952.

\bibitem{faudree97}
R.~Faudree, E.~Flandrin, and Z.~Ryj{\'a}{\v{c}}ek.
\newblock Claw-free graphs --- a survey.
\newblock {\em Discrete Mathematics}, 164(1--3):87--147, 1997.

\bibitem{georgakopoulos09}
A.~Georgakopoulos.
\newblock Infinite {Hamilton} cycles in squares of locally finite graphs.
\newblock {\em Advances in Mathematics}, 220(3):670--705, 2009.

\bibitem{gould03}
R.~J. Gould.
\newblock Advances on the {Hamiltonian} problem -- {A} survey.
\newblock {\em Graphs and Combinatorics}, 19(1):7--52, 2003.

\bibitem{gould14}
R.~J. Gould.
\newblock Recent advances on the {Hamiltonian} problem: {Survey III}.
\newblock {\em Graphs and Combinatorics}, 30(1):1--46, 2014.

\bibitem{hamann16}
M.~Hamann, F.~Lehner, and J.~Pott.
\newblock Extending cycles locally to {Hamilton} cycles.
\newblock {\em The Electronic Journal of Combinatorics}, 23(1), 2016.
\newblock \#P1.49.

\bibitem{heuer15}
K.~Heuer.
\newblock A sufficient condition for {Hamiltonicity} in locally finite graphs.
\newblock {\em European Journal of Combinatorics}, 45:97--114, 2015.

\bibitem{heuer16}
K.~Heuer.
\newblock A sufficient local degree condition for {Hamiltonicity} in locally
  finite claw-free graphs.
\newblock {\em European Journal of Combinatorics}, 55:82--99, 2016.

\bibitem{haggkvist81}
R.~Häggkvist and G.~G. Nicoghossian.
\newblock A remark on {Hamiltonian} cycles.
\newblock {\em Journal of Combinatorial Theory, Series B}, 30:118--120, 1981.

\bibitem{jung78}
H.~A. Jung.
\newblock On maximal circuits in finite graphs.
\newblock In B.~Bollobás, editor, {\em Advances in Graph Theory}, volume~3 of
  {\em Annals of Discrete Mathematics}, pages 129--144. Elsevier, 1978.

\bibitem{khachatrian85}
N.~K. Khachatryan.
\newblock Two methods of recognizing {Hamiltonicity} of a graph.
\newblock {\em Abstracts of papers submitted to the second all-union conference
  “Mathematical methods of image recognizing”}, pages 182--183, 1985.

\bibitem{kundgen17}
A.~Kündgen, B.~Li, and C.~Thomassen.
\newblock Cycles through all finite vertex sets in infinite graphs.
\newblock {\em European Journal of Combinatorics}, 65:259--275, 2017.

\bibitem{linial93}
N.~Linial.
\newblock Local-global phenomena in graphs.
\newblock {\em Combinatorics, Probability and Computing}, 2(4):491--503, 1993.

\bibitem{moon63}
J.~Moon and L.~Moser.
\newblock On {Hamiltonian} bipartite graphs.
\newblock {\em Israel Journal of Mathematics}, 1(3):163--165, 1963.

\bibitem{nara80}
C.~Nara.
\newblock On sufficient conditions for a graph to be {Hamiltonian}.
\newblock {\em Natural Science Report, Ochanomizu University}, 31(2):75--80,
  1980.

\bibitem{nash-williams71}
C.~{\relax St}. J.~A. Nash-Williams.
\newblock {Hamiltonian} arcs and circuits.
\newblock In M.~Capobianco, J.~B. Frechen, and M.~Krolik, editors, {\em Recent
  Trends in Graph Theory}, volume 186 of {\em Lecture Notes in Mathematics},
  pages 197--210. Springer, Berlin, Heidelberg, 1971.

\bibitem{oberly79}
D.~J. Oberly and D.~P. Sumner.
\newblock Every connected, locally connected nontrivial graph with no induced
  claw is {Hamiltonian}.
\newblock {\em Journal of Graph Theory}, 3(4):351--356, 1979.

\bibitem{ore60}
O.~Ore.
\newblock Note on {Hamilton} circuits.
\newblock {\em The American Mathematical Monthly}, 67(1):55, 1960.

\bibitem{pareek83}
C.~M. Pareek and Z.~Skupień.
\newblock On the smallest non-{Hamiltonian} locally {Hamiltonian} graph.
\newblock {\em J. Univ. Kuwait (Sci.)}, 10(1):9--17, 1983.

\bibitem{vanaardt16}
S.~A. van Aardt, M.~Frick, O.~R. Oellermann, and J.~P. de~Wet.
\newblock Global cycle properties in locally connected, locally traceable and
  locally hamiltonian graphs.
\newblock {\em Discrete Applied Mathematics}, 205:171--179, 2016.

\bibitem{vizing64}
V.~G. Vizing.
\newblock On an estimate of the chromatic class of a {$p$}-graph.
\newblock {\em Diskret. Analiz}, 3:25--30, 1964.

\end{thebibliography}
\end{document}